\documentclass{amsart}
\usepackage{amsfonts}
\usepackage{amssymb}
\usepackage{amsmath}
\usepackage{amsthm}
\theoremstyle{plain}
\newtheorem{thm}{Theorem}[section]

\theoremstyle{definition}

\newtheorem{ex}[thm]{Example}
\newtheorem{prop}[thm]{Proposition}

\numberwithin{equation}{section}

\begin{document}

\title[Elliptic systems in exterior domains]{ Coupled elliptic systems   depending on the gradient  
with nonlocal BCs in exterior domains}
\author[F. Cianciaruso]{Filomena Cianciaruso}%
\address{Filomena Cianciaruso, Dipartimento di Matematica e Informatica, Universit\`{a}
della Calabria, 87036 Arcavacata di Rende, Cosenza, Italy}
\email{filomena.cianciaruso@unical.it}
\author[L. Muglia]{Luigi Muglia}%
\address{Luigi Muglia, Dipartimento di Matematica e Informatica, Universit\`{a} della
Calabria, 87036 Arcavacata di Rende, Cosenza, Italy}
\email{muglia@mat.unical.it}
\author[P. Pietramala]{Paolamaria Pietramala}%
\address{Paolamaria Pietramala, Dipartimento di Matematica e Informatica, Universit\`{a}
della Calabria, 87036 Arcavacata di Rende, Cosenza, Italy}
\email{pietramala@unical.it}
\begin{abstract}
We study existence and multiplicity of positive radial solutions for a coupled elliptic system in exterior domains where the  nonlinearities depend  on the gradients and the boundary conditions are nonlocal. We use a nonstandard cone to establish existence of solutions by means of  fixed point index theory. 
\end{abstract}
\maketitle
\section{Introduction}
In this paper, we study existence and multiplicity of  positive radial solutions for the coupled elliptic system
\begin{equation} \label{do}
\begin{cases}
-\Delta u= h_1(|x|) f_1(u,v,|\nabla u|,|\nabla v|),\,\,\,\,\, |x|\in [r_0,+\infty), \\
-\Delta v= h_2(|x|) f_2(u,v,|\nabla u|,|\nabla v|),\,\,\,\,\, |x|\in [r_0,+\infty), \\
\displaystyle \lim_{|x|\to\infty}u(x)=\alpha_{1}[u],\, \,\,\,\,\,\displaystyle c_1u+\tilde{d_1}\frac{\partial u}{\partial r}=\beta_{1}[u] \,\,\text{ for }|x|=r_0,\\
\displaystyle\lim_{|x|\to\infty}v(x)=\alpha_{2}[v],\,\,\,\,\,\,\displaystyle c_2v+\tilde{d_2}\frac{\partial v}{\partial r}=\beta_{2}[v] \,\,\text{ for }|x|=r_0,
\end{cases}
\end{equation}
where $\alpha_i[\cdot],\, \beta_i[\cdot]$ are bounded linear functionals, $h_i$ and $f_i$ are nonnegative functions, $c_i\geq 0,\,\tilde{d}_i \leq 0$, $r_0>0$ and $\dfrac{\partial}{\partial r}$ denotes (as
in ~\cite{nirenberg}) differentiation in the radial direction
$r=|x|$. The functions $f_i$ are continuous and every singularity is captured by the term $h_i\in L^1$ that may have pointwise singularities.\\
Many papers are aimed to the existence of radial solutions of elliptic equations in the exterior part of a ball. A variety of methods has been used,  for instance, when  the boundary conditions (BCs) are \emph{homogeneous},  a priori estimates were utilized by Castro et al.~\cite{shi5},  sub and super solutions were used by Djedali and Orpel~\cite{orpel3} and Sankar et al.~\cite{shi4};  variational methods were used by Orpel~\cite{orpel2} and topological methods where employed by  Abebe and co-authors~\cite{abeshi}, do \'O et al.~\cite{do1}, Hai and Shivaji \cite{hai-shi}, Han and Wang~\cite{han}, Lee~\cite{lee1}, Orpel~\cite{orpel} and Stanczy~\cite{sta}.\\

In particular, in a very recent paper \cite{hai-shi}, Hai and Shivaji  proved existence and multiplicity of positive radial solutions for the superlinear ellptic  system 
\begin{equation*}\label{shi-syst}
\begin{cases}
-\Delta u=\lambda h_1(|x|) f_1(v), &|x|\in [r_0,+\infty), \\
-\Delta v=\lambda h_2(|x|) f_2(u), &|x|\in [r_0,+\infty), \\
\displaystyle \lim_{|x|\to\infty}u(x)=0,\, \,\,\displaystyle d_1\frac{\partial u}{\partial n}+\tilde{c}_1(u)u=0 \,\,\text{ for }&|x|=r_0,\\
\displaystyle\lim_{|x|\to\infty}v(x)=0,\,\,\,d_2\frac{\partial v}{\partial n}+\tilde{c}_2(v)v=0   \,\,\text{ for }&|x|=r_0,
\end{cases}
\end{equation*}
using  a fixed point result of Krasnoselskii type applied to a suitable completely continuous integral operators on $C[0,1]\times C[0,1]$. These results seems to be the first ones proving multiplicity of positive solutions for this kind of systems. \\
On the other hand, in the context of non-homogeneous BCs, elliptic problems were studied by Aftalion and Busca ~\cite{busca},  Butler and others~\cite{but}, Cianciaruso and co-authors in \cite{genupa},  Dhanya et al.~\cite{shi2}, do \'O et al.~\cite{do2, do4, do6, do7},  Goodrich~\cite{Goodrich3, Goodrich4}, Ko and co-authors ~\cite{shi3} and Lee and others~\cite{shi1}.\\
The existence of positive radial solutions of elliptic equations with nonlinearities depending on the gradient subject to Neumann, Dirichlet or Robin
 boundary conditions has been investigated by a number of authors,  see for example Averna \cite{ave-mo-tor}, Cianciaruso and co-authors \cite{nu, nu2, nupa2}, De Figueiredo and co-authors \cite{defig-sa-ubi, defig-ubi} and Faria and co-authors \cite{fa-mi-pe}. \\
Our system \eqref{do} is really general, indeed
\begin{enumerate}
\item the nonlinearities $f_i$ as depend on the functions $u$ and $v$ as depend on their gradients; no monotonicity hypotheses are supposed about them.
\item the boundary conditions are nonlocal and can be read as some kind of feedback mechanisms and they have been deeply studied for ordinary differential equations (see, for example ~\cite{nupa, goodrich1, goodrich2, webb2, webb3}).
\end{enumerate}
In order to search for solutions of an elliptic PDE
$$
-\Delta w=g(|x|)f(w, |\nabla w|)
$$
with some boundary conditions, a topological approach is to associate, by using standard transformations, an integral operator as
$$
Sw(t)=\int_0^1 G(t,s)g(r(s))\tilde{f}(w(s), |w'(s)|)ds.$$ 
It is straightforward, in the local problems, to find the Green's function $G$ by integration and by using the BCs. However let us remark that, in the nonlocal problems, this is a long and technical calculation whose result is often a sum of terms of different signs. \\
Here, as in \cite{webb2}, we treat the nonlocal problem as the perturbation of the simpler local problem. In such a way we handle with the positivity properties of the simpler Green's function of the local problem.\\
Often the associated integral operator is studied in the cone of nonnegative functions in the space $C^1[0,1]$ or in a weighed space of differential functions as in \cite{aga-ore-yan}.  In our case and in particular when seeking for multiple solutions, it is suitable to set in a smaller cone: we will introduce a new cone in which Harnack-type inequalities are used.\\
Moreover, since we are interested to positive solutions, the functionals $\alpha_i$ and $\beta_i$ must satisfy some positivity conditions; we will not suppose this  in the whole space, but we choose to include the  requirement in the definition of the cone.\\
We show that, under suitable conditions on the nonlinear terms, the fixed point  index is $0$ on certain open bounded subsets of the cone and $1$ on others; the choice of these subsets allows us to have more freedom on the conditions of the growth of the nonlinearities. 
These conditions relate the upper and lower bounds of the nonlinearities $f_i$ on
stripes and some constants, depending on the kernel of the 
integral operator and on the nonlocal BCs, that are easy estimable  as we show in a example.  
\section{The associate integral operator}

Consider  in $\mathbb{R}^n$, $n\ge 3$, the equation
\begin{equation}\label{eqell}
-\triangle w= h(|x|)f(w,|\nabla w|), \quad |x|\in
[r_{0},+\infty).
\end{equation}
Since we are interested to radial solutions $w=w(r)$, $r=|x|$, following \cite{but}, 
 we rewrite \eqref{eqell} as
\begin{equation}\label{eqinterm}
-w''(r) - \dfrac{n-1}{r}w'(r)= h(r)f(w(r),|w'(r)|),  \,\,\,\,  r\in [r_0, +\infty).
\end{equation}

By using the transformation
\begin{equation*}
r(t):=r_0\,t^{\frac{1}{2-n}},\,\,\,\,\,\,t\in(0,1],
\end{equation*}
equation (\ref{eqinterm}) becomes 
\begin{equation*}
 w''(r(t))+g(t)\,f\left(w(r(t)), \frac{|w'(r(t))|}{|r'(t)|}\right)=0, \qquad t\in(0,1],
\end{equation*}
with $$g(t)=\displaystyle \frac{r_0^2}{(n-2)^2}t^{\frac{2n-2}{2-n}} h(r(t)).$$

Consider in $\mathbb{R}^n$ the system of boundary value problems
\begin{equation} \label{do2}
\begin{cases}
-\Delta u= h_1(|x|) f_1(u,v,|\nabla u|,|\nabla v|),\,\,\,\,\, |x|\in [r_0,+\infty), \\
-\Delta v= h_2(|x|) f_2(u,v,|\nabla u|,|\nabla v|),\,\,\,\,\, |x|\in [r_0,+\infty), \\
\displaystyle \lim_{|x|\to\infty}u(x)=\alpha_{1}[u],\, \,\,\,\,\,\displaystyle c_1u+\tilde{d_1}\frac{\partial u}{\partial r}=\beta_{1}[u] \,\,\text{ for }|x|=r_0,\\
\displaystyle\lim_{|x|\to\infty}v(x)=\alpha_{2}[v],\,\,\,\,\,\,\displaystyle c_2v+\tilde{d_2}\frac{\partial v}{\partial r}=\beta_{2}[v] \,\,\text{ for }|x|=r_0.
\end{cases}
\end{equation}

Set
$u(t)=u(r(t))$ and $v(t)=v(r(t))$. Thus to the system \eqref{do2} we associate the system of ODEs
\begin{equation}\label{1syst}
\begin{cases}
u''(t) + g_1(t) f_1\left(u(t),v(t),\displaystyle\frac{|u'(t)|}{|r'(t)|},\displaystyle\frac{|v'(t)|}{|r'(t)|}\right) = 0,\,\,\,\, t\in (0,1), \\
v''(t) + g_2(t) f_2\left(u(t),v(t),\displaystyle\frac{|u'(t)|}{|r'(t)|},\displaystyle\frac{|v'(t)|}{|r'(t)|}\right) = 0, \,\,\,\,t\in  (0,1),\\
u(0)=\alpha_1[u],\,\,\, c_1u(1)+d_1 u'(1)=\beta_{1}[u],  \\
v(0)=\alpha_{2}[v],\,\,\, c_2v(1)+d_2 v'(1)=\beta_2[v],
\end{cases}
\end{equation}
where $ \displaystyle d_{i}=\frac{r_0}{2-n}\tilde{d}_{i}$.

We study the existence of positive solutions of the system~\eqref{1syst}, by means of the associated system of perturbed Hammerstein integral equations
\begin{equation*}\label{intrsyst}
\begin{cases}
 \displaystyle u(t)=\gamma_{1}(t)\alpha_{1}[u]+\delta_{1}(t)\beta_{1}[u]+ \int_{0}^{1}k_1(t,s)g_1(s)f_1\left(u(s),v(s),\frac{|u'(s)|}{|r'(s)|},\frac{|v'(s)|}{|r'(s)|}\right)ds, \\
\displaystyle v(t)=\gamma_{2}(t)\alpha_{2}[v]+\delta_{2}(t) \beta_{2}[v]+\int_{0}^{1}k_2(t,s)g_2(s)f_2\left(u(s),v(s),\frac{|u'(s)|}{|r'(s)|},\frac{|v'(s)|}{|r'(s)|}\right)ds,%
\end{cases}
\end{equation*}
where $\gamma_{i}$ is the solution of the BVP $w''(t)=0,\,\, w(0)=1, \, c_iw(1)+d_i w'(1)=0$, namely
$$
\gamma_{i}(t)=\displaystyle 1-\frac{c_{i}t}{d_i+c_{i} };
$$
$\delta_{i}$ is the solution of the BVP $w''(t)=0,\,\, w(0)=0, \,\,c_iw(1)+d_i w'(1)=1$, namely
$$
\delta_{i}(t)=\displaystyle \frac{t}{d_i+c_{i} }
$$
and $k_i$ is the Green's function associate to homogeneous problem, in wich  $\alpha_i[w]=\beta_i[w]=0$, namely 
\begin{equation*}\label{ki}
k_i(t,s):=
\begin{cases}\displaystyle 
s\left(1-\frac{c_{i} t}{d_i+c_{i}}\right), &  s \le t,\cr
 \displaystyle
 t\left(1-\frac{c_{i} s}{d_i+c_{i}}\right),&s>t.
\end{cases}
\end{equation*}
In the following Proposition we resume the properties of the functions $\gamma_i,\,\delta_i$ and $k_i$  that will be useful in the sequel.
\begin{prop} We have, for $i=1,2$,
\begin{itemize}
\item  The functions $\gamma_i,\,\delta_i$ are in $C^1 [0,1] $; moreover, for $t \in[a_i,b_i] \subset(0,1)$, with $a_i+b_i<1$,
 $$\|\gamma_{i}\|_\infty=1\text{ and }\gamma_{i}(t) \geq 1-t\geq1-b_i= (1-b_i)\|\gamma_{i}\|_\infty >a_i\|\gamma_{i}\|_\infty;$$
 $$\|\delta_{i}\|_\infty=\displaystyle\frac{1}{d_i+c_{i}}\,\,\, \text{ and  }\delta_{i}(t)=t\|\delta_{i}\|_\infty\geq  a_i\|\delta_{i}\|_\infty.$$
\item The kernels $k_i$ are nonnegative and continuous in  $[0,1]\times  [0,1]$. Moreover, for $t\in [a_i,b_i] $, we have
\begin{align*}
&k_i(t,s) \leq \phi_{i}(s)\ \text{ for } (t,s)\in [0,1] \times [0,1] \text{ and }\\
&k_i(t,s) \geq a_i\, \phi_{i}(s) \text{ for }(t,s)\in
[a_i,b_i]\times [0,1], 
\end{align*}
with
$\phi_{i}(s):=s\displaystyle\left(1-\frac{c_{i} }{d_i+c_{i}}s\right)$.
\end{itemize}
 \end{prop}
\smallskip

\noindent
Let  $\omega(t)=t(1-t)$. Our setting will be  the Banach space (see \cite{aga-ore-yan}) 
$$
C_{\omega}^1[0,1]=\{w \in C[0,1]\cap C^1(0,1) :\,\,\sup_{t \in (0,1)}\omega(t)|w'(t)|<+\infty\}
$$
endowed with the norm
$$
|| w||: =\max \left\{ || w|| _{\infty},\|w'\|_{\omega}\right\},
$$
where $|| w|| _{\infty}:=\underset{t\in [ 0,\,1]\,}{\max }|w(t)|$ and $\|w'\|_{\omega}:=\displaystyle\sup_{t\in (0,1)}\omega(t)|w'(t)|$.\\

For $i=1,2$, fixed $[a_i,b_i] \subset(0,1)$ such that $a_i+b_i<1$, we  consider the
 cones 
\begin{equation*}
K_{i}:=\left\{ w\in C_{\omega}^{1}[0,\,1]:w\geq 0,\,\underset{t\in [ a_i,b_i]}%
{\min }w(t)\geq a_i || w|| _{\infty},\,\,   \|w'\|_{\omega}\leq4w(1/2),\,\,\alpha_i[w]\geq 0,\,\beta_i[w]\geq 0\right\},
\end{equation*}
in $C_{\omega}^{1}[0,1]$ and the cone $$K:=K_{1}\times K_{2}$$ in $C_{\omega}^{1}[0,1]\times C_{\omega}^{1}[0,1]$.\\
Note that the functions in $K_{i}$ are strictly positive on the
sub-interval $[a_i,b_i]$ and that for $w \in K_{i}$ the inequalities $\|w\|_{\omega}\leq\|w\|\le 4 \|w\|_{\infty}$ hold.\\

From now on, we will assume that, for $i=1,2$,
\begin{itemize}
\item $f_i:[0,+\infty)\times [0,+\infty)\times\mathbb{R}\times \mathbb{R} \to[0,+\infty)$  is continuous;
\item  $h_i:[r_0,+\infty) \to [0,+\infty)$ is continuous  and $h_i(r)\leq \frac{1}{r^{n+\mu_i}}$ for $r\to +\infty$ for some $\mu_i>0$.
\item $0\leq \alpha_i[\gamma_i]<1$ \text{ and }$0\leq \beta_i[\delta_i]<1$;
 \item $\alpha_i[k_i]:= \alpha_i[k_i(\cdot,s)] \text{ and }\beta_i[k_i]:=\beta_i[k_i(\cdot,s)]$ are nonnegative numbers;
\item $D_i=(1-\alpha_i[\gamma_i])(1-\beta_i[\delta_i])-\alpha_i[\delta_i]\beta_i[\gamma_i]>0$.
\end{itemize}

Set
\begin{equation*}
F_i(u,v)(t):=\int_{0}^{1}k_i(t,s)g_i(s)f_i\left(u(s),v(s),\frac{|u'(s)|}{|r'(s)|},\frac{|v'(s)|}{|r'(s)|}\right)ds,
\end{equation*}
where $g_i(t)=\displaystyle \frac{r_0^2}{(n-2)^2}t^{\frac{2n-2}{2-n}} h_i(r(t))$;
the well-defined integral operator $T:C_{\omega}^{1}[0,1]\times C_{\omega}^{1}[0,1]\to C_{\omega}^{1}[0,1]\times C_{\omega}^{1}[0,1]$ given by 
\begin{gather}
\begin{aligned}\label{opT}
T(u,v)(t):=
\left(
\begin{array}{c}
\gamma_{1}(t)\alpha_{1}[u]+\delta_{1}(t)\beta_{1}[u]+F_1(u,v)(t) \\
\gamma_{2}(t)\alpha_{2}[v]+\delta_{2}(t) \beta_{2}[v]+F_2(u,v)(t)%
\end{array}
\right)
:=
\left(
\begin{array}{c}
T_1(u,v)(t) \\
T_2(u,v)(t)%
\end{array}
\right) 
\end{aligned}
\end{gather}
leaves the cone $K$ invariant and it is completely continuous.
\begin{thm}
The operator $T$ maps $K$ in $K$ and is completely continuous.
\end{thm}
\begin{proof}
In order to prove that $T$  leaves the cone $K$ invariant, it is enough to prove that $T_{i}K\subset K_{i}$. 

Take
$(u_1,u_2)\in K$ such that $\|(u_1,u_2)\| \leq r$, $r>0$. We have, for $t\in [0,1]$,
\begin{align*}
\|T_i(u_1,u_2)\|_\infty& \leq\|\gamma_i\|_\infty\alpha_i[u_i]+\|\delta_i\|_\infty\beta_i[u_i]+\\
&+\int_{0}^{1}\phi_i(s)g_i(s)f_i\left(u_1(s),u_2(s),\frac{|u_1'(s)|}{|r'(s)|},\frac{|u'_2(s)|}{|r'(s)|}\right)ds<+\infty.
\end{align*}
On the other hand, we have
\begin{align*}\label{min}
\min_{t\in [a_i,b_i]}T_i(u_1,u_2)(t) &\geq a_i\left(\|\gamma_i\|_\infty\alpha_i[u_i]+\|\delta_i\|_\infty\beta_i[u_i]+\right.\ \\
&+\left.\int_{0}^{1}\phi_i(s)g_i(s)f_i\left(u_1(s),u_2(s),\frac{|u_1'(s)|}{|r'(s)|},\frac{|u_2'(s)|}{|r'(s)|}\right)\,ds \right)\\
&\geq a_i ||T_i(u_1,u_2)|| _{\infty}.
\end{align*}
Now  we prove that for every $(u_1,u_2)\in  K$ 
$$\|(F_i(u_1,u_2))'\|_{\omega}\le 4 F_i(u_1,u_2)(1/2).$$ We have
\begin{align*}
\omega(t)|(F_i(u_1,u_2&))'(t)|=\Big|-t(1-t)\displaystyle\int_0^t
\frac{c_{i} s}{d_i+c_{i}}g_i(s)f_i\left(u_1(s),u_2(s),\frac{|u_1'(s)|}{|r'(s)|},\frac{|u_2'(s)|}{|r'(s)|}\right)ds\\
&+t(1-t)\displaystyle\int_t^1
 \left(1-\frac{c_{i} s}{d_i+c_{i}}\right)g_i(s)f_i\left(u_1(s),u_2(s),\frac{|u_1'(s)|}{|r'(s)|},\frac{|u_2'(s)|}{|r'(s)|}\right)ds\Big|\\
&\leq\displaystyle\int_0^t
t(1-t) \frac{c_{i} s}{d_1+c_{i}}g_i(s)f_i\left(u_1(s),u_2(s),\frac{|u_1'(s)|}{|r'(s)|},\frac{|u_2'(s)|}{|r'(s)|}\right)ds\\
&+\displaystyle\int_t^1
 t(1-t)\left(1-\frac{c_{i} s}{d_i+c_{i}}\right)g_i(s)f_i\left(u_1(s),u_2(s),\frac{|u_1'(s)|}{|r'(s)|},\frac{|u_2'(s)|}{|r'(s)|}\right)ds\\
\end{align*}
Since $\displaystyle \frac{c_{i} t}{d_i+c_{i}}\leq t$, it holds that $\displaystyle (1-t)\leq \left(1-\frac{c_{i} t}{d_i+c_{i}}\right)$
and consequently 
\begin{align*}
\omega(t)|(F_i(u_1,u_2))'(t)|&\leq\int_0^t
s\left(1-\frac{c_{i} t}{d_i+c_{i}}\right)g_i(s)f_i\left(u_1(s),u_2(s),\frac{|u_1'(s)|}{|r'(s)|},\frac{|u_2'(s)|}{|r'(s)|}\right)ds\\
&+\int_t^1 t\left (1-\frac{c_{i} s}{d_i+c_{i}}\right)g_i(s)f_i\left(u_1(s),u_2(s),\frac{|u_1'(s)|}{|r'(s)|},\frac{|u_2'(s)|}{|r'(s)|}\right)ds\\
&=\int_0^1k_i(t,s)g_i(s)f_i\left(u_1(s),u_2(s),\frac{|u_1'(s)|}{|r'(s)|},\frac{|u_2'(s)|}{|r'(s)|}\right)ds\\
&=F_i(u_1,u_2)(t)\le \|F_i(u_1,u_2)\|_{\infty}.
\end{align*}
Let $\tau_i \in [0,1]$ be such that
$$
 F_i(u_1,u_2)(\tau_i)=\|F_i(u_1,u_2)\|_\infty.
$$
For any $t\in [0,1]$ we can easily compute that
\begin{equation*} \frac{k_i(t,s)}{k_i(\tau_i,s)}=
\begin{cases}
\displaystyle t/\tau_i, &\tau_i,t\leq s, \\ \\
\displaystyle \left(1-\frac{c_{i} t}{d_i+c_{i}}\right)\left(1-\frac{c_{i} \tau_i}{d_i+c_{i}}\right)^{-1}, &\tau_i,t> s, \\ \\
\displaystyle t\left(1-\frac{c_{i} s}{d_i+c_{i}}\right)s^{-1}\left(1-\frac{c_{i} \tau_i}{d_i+c_{i}}\right)^{-1}, &t\leq s\leq\tau_i, \\ \\
\displaystyle s\left(1-\frac{c_{i} t}{d_i+c_{i}}\right)\tau_i^{-1}\left(1-\frac{c_{i} s}{d_i+c_{i}}\right)^{-1}, &	\tau_i\leq s\leq t \\ \\
\end{cases}
\end{equation*}
and that $\displaystyle\frac{k_i(t,s)}{k_i(\tau_i,s)}\geq t(1-t)$ for $t,s\in [0,1]$.
Then, for all $t\in [0,1]$, we have
\begin{eqnarray*}
F_i(u_1,u_2)(t)&=&\int_0^1 \frac{k_i(t,s)}{k_i(\tau_i,s)}k_i(\tau_i,s)g_i(s)f_i\left(u_1(s),u_2(s),\frac{|u_1'(s)|}{|r'(s)|},\frac{|u_2'(s)|}{|r'(s)|}\right)ds \ \\
&\geq& t(1-t)\int_0^1k_i(\tau_i,s)g_i(s)f_i\left(u_1(s),u_2(s),\frac{|u_1'(s)|}{|r'(s)|},\frac{|u_2'(s)|}{|r'(s)|}\right)ds \ \\
&=&t(1-t)\|F_i(u_1,u_2)\|_{\infty}
\end{eqnarray*}
For $t=\frac{1}{2}$ we obtain
\begin{equation*}
\|F_i(u_1,u_2)\|_{\infty}\leq 4F_{i}(u_1,u_2)\left(\frac{1}{2}\right)
\end{equation*}
Let us consider 
\begin{align*}
 \omega(t)|(F_i(&u_1,u_2))'(t)|=\omega(t)\Big|\int_0^t\frac{-c_{i} s}{d_i+c_{i}}g_i(s)f_i\left(u_1(s),u_2(s),\frac{|u_1'(s)|}{|r'(s)|},\frac{|u_2'(s)|}{|r'(s)|}\right)ds\\
&+\int_t^1\left(1-\frac{c_{i} s}{d_i+c_{i}}\right)g_i(s)f_i\left(u_1(s),u_2(s),\frac{|u_1'(s)|}{|r'(s)|},\frac{|u_2'(s)|}{|r'(s)|}\right)ds\Big|\\
 &\leq\,t(1-t)\int_0^t\frac{c_{i} }{d_i+c_{i}}sg_i(s)f_i\left(u_1(s),u_2(s),\frac{|u_1'(s)|}{|r'(s)|},\frac{|u_2'(s)|}{|r'(s)|}\right)ds\\
&+t(1-t)\int_t^1\left(1-\frac{c_{i}s }{d_i+c_{i}}\right)s\,g_i(s)f_i\left(u_1(s),u_2(s),\frac{|u_1'(s)|}{|r'(s)|},\frac{|u_2'(s)|}{|r'(s)|}\right)ds\\
&\leq\left(1-\frac{c_{i} t}{d_i+c_{i}}\right)\int_0^ts\,g_i(s)f_i\left(u_1(s),u_2(s),\frac{|u_1'(s)|}{|r'(s)|},\frac{|u_2'(s)|}{|r'(s)|}\right)ds\\
&+t\int_t^1\left(1-\frac{c_{i} s}{d_i+c_{i}}\right)g_i(s)f_i\left(u_1(s),u_2(s),\frac{|u_1'(s)|}{|r'(s)|},\frac{|u_2'(s)|}{|r'(s)|}\right)ds\\
&\leq\int_0^tk_i(t,s)g_i(s)f_i\left(u_1(s),u_2(s),\frac{|u_1'(s)|}{|r'(s)|},\frac{|u_2'(s)|}{|r'(s)|}\right)ds\\
&+\int_t^1k_i(t,s)g_i(s)f_i\left(u_1(s),u_2(s),\frac{|u_1'(s)|}{|r'(s)|},\frac{|u_2'(s)|}{|r'(s)|}\right)ds\\
&=F_i(u_1,u_2)(t)\leq \|F_i(u_1,u_2)\|_\infty
\end{align*}
Therefore we conclude that
$$
\|(F_i(u_1,u_2))'\|_{\omega}\le 4 F_i(u_1,u_2)\left(\frac{1}{2}\right).
$$
Since $\gamma_i, \delta_i \in K_i$, we have
\begin{eqnarray*}
|\omega(t)(T_i(u_1,u_2))'(t)|&\leq& \omega(t)|\gamma'_{i}(t)|\alpha_{i}[u_i]+\omega(t)|\delta'_{i}(t)|\beta_{i}[u_i]\\
&&+\omega(t)|(F_i(u_i,u_2))'(t)|\\
&\leq&\|\gamma'_i\|_{\omega}\alpha_{i}[u_i]+\|\delta'_i\|_{\omega}\beta_{i}[u_i]+\|F'_i(u_1,u_2)\|_{\omega}\\
&\leq&4\gamma_i\left(\frac{1}{2}\right)\alpha_i[u_i]+4\delta_i\left(\frac{1}{2}\right)\beta_i[u_i]+4F_i(u_i,u_2)\left(\frac{1}{2}\right)\\
&=&4T_i(u_1,u_2)\left(\frac{1}{2}\right)
\end{eqnarray*}
Taking the supremum on $[0,1]$, we obtain
$$
\|(T_i(u_1,u_2))'\|_{\omega}\le 4 T_i(u_1,u_2)\left(\frac{1}{2}\right).
$$
 Since $\alpha_i$ and $\beta_i$ are linear functionals it follows that $\alpha_i[T_i(u_1,u_2)]$ and $\beta_i[T_i(u_1,u_2)]$ are nonnegative.\\ Summarizing we have $TK\subset K$.

In order to prove the completely continuity of $T$, let us note that the continuity of $f$, $k_i$, $\alpha_{i}$ and $\beta_i$  give the continuity of each $T_i$ and then the continuity of $T$.

Let $U$ be a bounded subset of $K$; we prove that $T(U)$ is relatively compact in $K$. It is a standard argument based the uniform continuity of the kernels $k_i$ on $[0,1]\times [0,1]$ and on Ascoli-Arzel\'a Theorem that $T_i(U)$ is relatively compact in $C[0,1]$.\\
Now let $(u_n,v_n)_{n\in\mathbb{N}}$ a sequence in $U$; then $T_i(u_n,v_n)\subset K_i$. 

There exists $(u_{n_k},v_{n_k})_{k\in\mathbb{N}}$ such that $(T_1(u_{n_k},v_{n_k}))_{k\in \mathbb{N}}$ converges in $C[0,1]$.

Since $T_2(U)$ is relative compact, there exists  $(u_{n_{k_p}},v_{n_{k_p}})_{p\in\mathbb{N}}:=(u_{n_p},v_{n_p})_{p\in\mathbb{N}}\subset  (u_{n_k},v_{n_k})_{k\in\mathbb{N}}$ such that  $(T_i(u_{n_p},v_{n_p}))_{p\in \mathbb{N}}\to w_i\in C[0,1]$ for $i=1,2$.  Since 
$$
 \|(T_i(u_{n_p},v_{n_p}))'-(T_i(u_{n_m},v_{n_m}))'\|_{\omega}\leq \|T_i(u_{n_p},v_{n_p})-T_i(u_{n_m},v_{n_m})\|_\infty
$$
i.e. $((T(u_{n_p},v_{n_p}))')_{p\in \mathbb{N}}$ is a Cauchy sequence in $\|\cdot\|_{\omega}$ it converges to $(w_1,w_2)$.  Closedness of $K$ implies that $(w_1,w_2)\in K$ therefore $T(U)$ is relatively compact in $C_{\omega}^1(0,1)$.

\end{proof}

In order to use the fixed point index,  we utilize the following sets in $K$, for
$\rho_1,\rho_2>0$,
\begin{equation*}
K_{\rho _{1},\rho _{2}}:=\{(u,v)\in K:|| u||_\infty <\rho _{1}\ \text{ and }\ || v||_\infty <\rho _{2}\},
\end{equation*}
\begin{equation*}
 V_{\rho _{1},\rho _{2}}:=\{(u,v)\in K:
\displaystyle{\min_{t\in [a_1,b_1]}}u(t)<\rho_1 \text{ and
}\displaystyle{\min_{t\in [a_2,b_2]}}v(t)<\rho_2\}.
\end{equation*}
Since $\|w'\|_\omega\leq 4\|w\|_\infty$ in $K$, then $\|w\|\leq 4\|w\|_\infty$ therefore $K_{\rho_1,\rho_2}$ and $V_{\rho_1,\rho_2}$ are open and bounded relative to $K$.
It is straightforward  to verify that the above sets satisfy the properties:
\begin{itemize}
\item[$(P_1)$]$K_{\rho_1,\rho_2}\subset V_{\rho_1,\rho_2}\subset K_{\rho_1/a_1,\rho_2/a_2}$.
\item[$(P_2)$] $(w_1,w_2)\in\partial K_{\rho _{1},\rho _{2}}$ if and only if
$(w_{1},w_{2})\in K$ and for some $i\in \{1,2\}$  $\|w_i\|_{\infty}=\rho_i$ and $a_i \rho_i \le
w_i(t)\le \rho_i$ for $t\in[a_i,b_i]$.
\item[$(P_3)$] $(w_1,w_2) \in \partial V_{\rho_1,\rho_2}$ if and only if  $(w_1,w_2)\in K$ and for some $i\in \{1,2\}$ $\displaystyle
\min_{t\in [a_i,b_i]} w_i(t)= \rho_i$  and
 $\rho_i \le w_i(t) \le \rho_i/a_i$  for  $t\in[a_i,b_i]$.
\end{itemize}
\,\,\,\,
The following theorem follows from classical
results about fixed point index (more details can be seen, for example, in
\cite{Amann-rev, guolak}).
\begin{thm} \label{index}Let $K$ be a cone in an ordered Banach space $X$. Let $\Omega $ be
an open bounded subset with $0 \in \Omega\cap K$ and
$\overline{\Omega \cap K}\neq K$.  Let $\Omega ^{1}$ be open in
$X$ with $\overline{\Omega ^{1}}\subset \Omega \cap K$. Let
$F:\overline{\Omega \cap K}\rightarrow K$ be a compact map.
 Suppose that
\begin{itemize}
\item[(1)]$Fx\neq \mu x$ for all $x\in\partial( \Omega \cap K)$
and for all $\mu \geq 1$.
\item[(2)] There exists $h\in K\setminus \{0\}$ such that $x\neq Fx+\lambda h$ for all $x\in \partial (\Omega^1 \cap K)$ and all $\lambda
>0$.
\end{itemize}
Then $F$ has at least one fixed point $x \in (\Omega \cap
K)\setminus\overline{(\Omega^{1}\cap K)}$.\\
Denoting by $i_K(F,U)$ the fixed point index of $F$ in some
$U\subset X$, we have $$i_{K}(F,\Omega \cap K)=1 \mbox{ and }
i_{K}(F,\Omega^{1} \cap K)=0\,.$$
 The same result holds if
$$i_{K}(F,\Omega \cap K)=0 \mbox{ and }
i_{K}(F,\Omega^{1} \cap K)=1.$$

\end{thm}

\section{The system of elliptic PDE}

We define the following sets:
\begin{equation*}
\Omega^{\rho_1,\rho_2}=\left [0,
\rho_1\right]\times\left [0, \rho_2\right]\times
\left[0, +\infty\right)\times\left [ 0,
+\infty\right),
\end{equation*}
\begin{align*}
&A_1^{s_1,s_2}=
\left[s_1,\frac{s_1}{a_1}\right]\times\left[0,\frac{s_2}{a_2}\right]\times\left[0, +\infty\right)\times\left [ 0,
+\infty\right),\\
&A_2^{s_1,s_2}=\left[0,\frac{s_1}{a_1}\right]\times\left[s_2,\frac{s_2}{a_2}\right]\times\left[0, +\infty\right)\times\left [ 0,
+\infty\right),
\end{align*}
and the numbers
\begin{align*}
&\displaystyle C_i:=\left[\frac{1}{D_i}\Bigg(\bigl[(1-\beta_i[\delta_i])+\|\delta_i\|_{\infty}\beta_i[\gamma_i]\bigr]\int_0^1\alpha_i[k_i]g_i(s)ds\right. \\
&\left. \nonumber\displaystyle+\bigl[\alpha_i[\delta_i]+\|\delta_i\|_{\infty} (1-\alpha_i[\gamma_i])\bigr]\int_0^1\beta_i[k_i]g_i(s)ds\Bigg)
+\sup_{t\in [0,1]}\int_0^1k_i(t,s)g_i(s)ds\right]^{-1}\\
&M_i=\left[\frac{1}{D_i}\left(\Bigl[a_i(1-\beta_i[\delta_i])+a_i\|\delta_i\|_{\infty}\beta_i[\gamma_i]\Bigr]\int_{a_i}^{b_i}\alpha_i[k_i]g_i(s)ds\right.\right.\\
&\left.\left.+\Bigl[a_i\alpha_i[\delta_i]+a_i\|\delta_i\|_{\infty} (1-\alpha_i[\gamma_i])\Bigr]\int_{a_i}^{b_i}\beta_1[k_i]g_i(s)ds\right)+\inf_{t \in [a_i,b_i]}\int_{a_i}^{b_i}k_i(t,s)g_i(s)ds)\right]^{-1}.
\end{align*}
\begin{thm}\label{ellyptic}
 Suppose that there exist
$\rho_1, \rho_2, s_1,s_2\in (0,+\infty )$, with $\rho _{i}<\,s
_{i}\,,\,\,i=1,2$, such that the following conditions hold
\begin{equation}\label{pde}
\sup_{\Omega^{\rho_1,\rho_2}}
f_i(w_1,w_2,z_1,z_2)<C_i\rho_i
\end{equation}
and
\begin{equation}\label{Mi}
\inf_{A_i^{s_1,s_2}}
f_i(w_1,w_2,z_1,z_2)>M_is_i
\end{equation}
Then the system (\ref{do2}) has at least one positive radial
solution. 
\end{thm}
\begin{proof}
We note that the choice of the numbers $\rho_i$ and $s_i$
assures the compatibility of conditions \eqref{pde} and \eqref{Mi}.\\
We want to show that $i_K(T, K_{\rho_1,\rho_2})=1$ and $i_K(T,
V_{s_1,s_2})=0$. In such a way that from Theorem \ref{index} it follows that the completely continuous operator $T$ has a
fixed point in $V_{s_1,s_2}\setminus
\overline{K}_{\rho_1,\rho_2}$. Then the system \eqref{do2} admits a positive radial solution. \\
Firstly we claim that $\lambda (u,v)\neq T(u,v)$ for every $(u,v)\in
\partial K_{\rho_1,\rho_2}$ and for every $\lambda \geq 1$,
which implies that the index of $T$ is 1 on $K_{\rho_1,\rho_2}$.
Assume this is not true. 
 Let $\lambda\geq1$ and let $(u,v)\in\partial K_{\rho_1,\rho_2}$ such that
$$
 \lambda(u,v)=T(u,v).
$$
In view of $(P_2)$, without loss in generality, let us suppose that $\|u\|_{\infty}=\rho_1$.\\
Then
\begin{eqnarray}\label{star}
 \lambda u(t)=\gamma_1(t)\alpha_1[u]+\delta_1(t)\beta_1[u]+\int_0^1k_1(t,s)g_1(s)f_1\left(u(s),v(s),\frac{|u'(s)|}{|r'(s)|},\frac{|v'(s)|}{|r'(s)|}\right)ds
\end{eqnarray}
Applying $\alpha_1$ to both terms we have:
$$
\lambda\alpha_1[u]=\alpha_1[\gamma_1]\alpha_1[u]+\alpha_1[\delta_1]\beta_1[u]+\int_0^1\alpha_1[k_1]g_1(s)f_1\left(u(s),v(s),\frac{|u'(s)|}{|r'(s)|},\frac{|v'(s)|}{|r'(s)|}\right)ds
$$
that implies
$$
(\lambda-\alpha_1[\gamma_1])\alpha_1[u]-\alpha_1[\delta_1]\beta_1[u]=\int_0^1\alpha_1[k_1]g_1(s)f_1\left(u(s),v(s),\frac{|u'(s)|}{|r'(s)|},\frac{|v'(s)|}{|r'(s)|}\right)ds
$$
In a similar way, applying $\beta_1$ we obtain
$$
(\lambda-\beta_1[\delta_1])\beta_1[u]-\beta_1[\gamma_1]\alpha_1[u]=\int_0^1\beta_1[k_1]g_1(s)f_1\left(u(s),v(s),\frac{|u'(s)|}{|r'(s)|},\frac{|v'(s)|}{|r'(s)|}\right)ds\\
$$
Denoting by
$$
N_1^\lambda:=\left(
\begin{array}{ccccc}
\lambda-\alpha_1[\gamma_1] & -\alpha_1[\delta_1] \\
 -\beta_1[\gamma_1] & \lambda-\beta_1[\delta_1]
\end{array}
\right),\,\,\,\,\,\,N^1_1:=N_1\,\,\,\textit{ and }\,\,\,D_1:=\det N_1,
$$

previous conditions can written as
$$
N_1^\lambda\left(
\begin{array}{ccccc}
\alpha_1[u] \\
\beta_1[u]
\end{array}
\right)=
\left(
\begin{array}{ccccc}\displaystyle
\int_0^1\alpha_1[k_1]g_1(s)f_1\left(u(s),v(s),\frac{|u'(s)|}{|r'(s)|},\frac{|v'(s)|}{|r'(s)|}\right)ds\\
\displaystyle\int_0^1\beta_1[k_1]g_1(s)f_1\left(u(s),v(s),\frac{|u'(s)|}{|r'(s)|},\frac{|v'(s)|}{|r'(s)|}\right)ds
\end{array}
\right)
$$
therefore, we get that
\begin{eqnarray*}
\left(
\begin{array}{ccccc}
\alpha_1[u] \\
\beta_1[u]
\end{array}
\right)=(N_1^\lambda)^{-1}\left(
\begin{array}{ccccc}
\displaystyle\int_0^1\alpha_1[k_1]g_1(s)f_1\left(u(s),v(s),\frac{|u'(s)|}{|r'(s)|},\frac{|v'(s)|}{|r'(s)|}\right)ds\\
\displaystyle\int_0^1\beta_1[k_1]g_1(s)f_1\left(u(s),v(s),\frac{|u'(s)|}{|r'(s)|},\frac{|v'(s)|}{|r'(s)|}\right)ds
\end{array}
\right)\\
\leq (N_1)^{-1}\left(
\begin{array}{ccccc}
\displaystyle\int_0^1\alpha_1[k_1]g_1(s)f_1\left(u(s),v(s),\frac{|u'(s)|}{|r'(s)|},\frac{|v'(s)|}{|r'(s)|}\right)ds\\
\displaystyle\int_0^1\beta_1[k_1]g_1(s)f_1\left(u(s),v(s),\frac{|u'(s)|}{|r'(s)|},\frac{|v'(s)|}{|r'(s)|}\right)ds
\end{array}
\right)\\
\end{eqnarray*}
so that, formula (\ref{star}) becomes for $t \in [0,1]$
\begin{eqnarray*}
u(t)&\leq &\frac{1}{D_1}\left[\gamma_1(t)(1-\beta_1[\delta_1])\int_0^1\alpha_1[k_1]g_1(s)f_1\left(u(s),v(s),\frac{|u'(s)|}{|r'(s)|},\frac{|v'(s)|}{|r'(s)|}\right)ds \right.\\
&&\left. +\gamma_1(t)\alpha_1[\delta_1]\int_0^1\beta_1[k_1]g_1(s)f_1\left(u(s),v(s),\frac{|u'(s)|}{|r'(s)|},\frac{|v'(s)|}{|r'(s)|}\right)ds\right.\\
&&\left. +\delta_1(t)\beta_1[\gamma_1]\int_0^1\alpha_1[k_1]g_1(s)f_1\left(u(s),v(s),\frac{|u'(s)|}{|r'(s)|},\frac{|v'(s)|}{|r'(s)|}\right)ds \right.\\
&&\left.  +\delta_1(t) (1-\alpha_1[\gamma_1])\int_0^1\beta_1[k_1]g_1(s)f_1\left(u(s),v(s),\frac{|u'(s)|}{|r'(s)|},\frac{|v'(s)|}{|r'(s)|}\right)ds\right]\\
&&+\int_0^1k_1(t,s)g_1(s)f_1\left(u(s),v(s),\frac{|u'(s)|}{|r'(s)|},\frac{|v'(s)|}{|r'(s)|}\right)ds\\
&=&\frac{1}{D_1|}\left[\Bigl[\gamma_1(t)(1-\beta_1[\delta_1])+\delta_1(t)\beta_1[\gamma_1]\Bigr]\int_0^1\alpha_1[k_1]g_1(s)f_1\left(u(s),v(s),\frac{|u'(s)|}{|r'(s)|},\frac{|v'(s)|}{|r'(s)|}\right)ds \right.\\
&&+\left.\Bigl[\gamma_1(t)\alpha_1[\delta_1]+\delta_1(t) (1-\alpha_1[\gamma_1])\Bigr]\int_0^1\beta_1[k_1]g_1(s)f_1\left(u(s),v(s),\frac{|u'(s)|}{|r'(s)|},\frac{|v'(s)|}{|r'(s)|}\right)ds\right]\\
&&+\int_0^1k_1(t,s)g_1(s)f_1\left(u(s),v(s),\frac{|u'(s)|}{|r'(s)|},\frac{|v'(s)|}{|r'(s)|}\right)ds\\
&\leq&\sup_{\Omega^{\rho_1,\rho_2}}
f_1(w_1,w_2,z_1,z_2)\left[\frac{1}{D_1}\left[\Bigl(\gamma_1(t)(1-\beta_1[\delta_1])+\delta_1(t)\beta_1[\gamma_1]\Bigr)\int_0^1\alpha_1[k_1]g_1(s)ds \right.\right.\\
&&+\left.\Bigl(\gamma_1(t)\alpha_1[\delta_1]+\delta_1(t) (1-\alpha_1[\gamma_1])\Bigr)\int_0^1\beta_1[k_1]g_1(s)ds\right]\\
&&\left.+\int_0^1k_1(t,s)g_1(s)ds\right].\\\
\end{eqnarray*}

Taking the supremum in $[0,1]$ in the last inequality, it follows that 
\begin{eqnarray*}
\rho_1=||u||_{\infty}&\leq&\sup_{\Omega^{\rho_1,\rho_2}}
f_1(w_1,w_2,z_1,z_2)\left[\frac{1}{D_1}\Bigg[\Bigl[(1-\beta_1[\delta_1])+\|\delta_1\|_{\infty}\beta_1[\gamma_1]\Bigr]\int_0^1\alpha_1[k_1]g_1(s)ds\right. \\
\nonumber\displaystyle&+&\Bigl[\alpha_1[\delta_1]+\|\delta_1\|_{\infty} (1-\alpha_1[\gamma_1])\Bigr]\int_0^1\beta_1[k_1]g_1(s)ds\Bigg]\\
\nonumber\displaystyle &+&\sup_{t \in [0,1]}\int_0^1k_1(t,s)g_1(s)ds\Bigg]=\frac{1}{C_1}\sup_{\Omega^{\rho_1,\rho_2}}f_1(w_1,w_2,z_1,z_2)<\rho_1,
\end{eqnarray*}
 that is a contradiction.\\
Now we show that that the index of $T$ is $0$ on $V_{s_1,s_2}$.\\
Consider $l(t)=1$ for $t\in [ 0,1],$ and note that $(l,l)\in
K$.
Now we claim that
\begin{equation*}
(u,v)\neq T(u,v)+\lambda (l,l)\quad \text{for }(u,v)\in \partial V_{s_1,s_2}\quad \text{and }\lambda \geq 0.
\end{equation*}% 
Assume, by contradiction, that there exist $(u,v)\in \partial V_{s_1,s_2}$ and $\lambda \geq 0$ such that $(u,v)=T(u,v)+\lambda (l,l)$.\\
Without loss of generality, we can assume that $\displaystyle\min_{t\in[a_1,b_1]} u(t)=
s_1$ and $\,s_1\leq u(t)\leq {s_1/a_1}$  for $t\in [a_1,b_1]$.
Then, for $t\in [ a_{1},b_{1}]$, we obtain
\begin{equation}\label{u}
u(t) =\gamma_1(t)\alpha_1[u]+\delta_1(t)\beta_1[u]+\int_0^1k_1(t,s)g_1(s)f_1\left(u(s),v(s),\frac{|u'(s)|}{|r'(s)|},\frac{|v'(s)|}{|r'(s)|}\right)ds+\lambda. 
%&\geq\int_0^1k_1(t,s,u)g_1(s)f_1\left(u(s),v(s),\frac{|u'(s)|}{|r'(s)|},\frac{|v'(s)|}{|r'(s)|}\right)ds\\
%&\geq \inf_{A_1^{s_1,s_2}}f_1(w_1,w_2,z_1,z_2)\left(\inf_{(t,s)\in [a_1,b_1]\times \left[s_1, \frac{s_1}{a_1}\right]}\int_0^1k_1(t,s,u)g_1(s)ds\right)\\
%&>M_{s_1}s_1\left(\inf_{(t,s)\in [a_1,b_1]\times \left[s_1, \frac{s_1}{a_1}\right]}\int_0^1k_1(t,s,u)g_1(s)ds\right) =s_1.
\end{equation}
Applying $\alpha_1$ and $\beta_1$ to both sides of \eqref{u} gives
\begin{align*}
\alpha_1[u]&=\alpha_1[\gamma_1]\alpha_1[u]+\alpha_1[\delta_1]\beta_1[u]+\int_0^1\alpha_1[k_1]g_1(s)f_1\left(u(s),v(s),\frac{|u'(s)|}{|r'(s)|},\frac{|v'(s)|}{|r'(s)|}\right)ds+\lambda \alpha_1[1], \\
\beta_1[u]&=\beta_1[\gamma_1]\alpha_1[u]+\beta_1[\delta_1]\beta_1[u]+\int_0^1\beta_1[k_1]g_1(s)f_1\left(u(s),v(s),\frac{|u'(s)|}{|r'(s)|},\frac{|v'(s)|}{|r'(s)|}\right)ds+\lambda \beta_1[1].
\end{align*}
Thus, we have
\begin{align*}
&(1-\alpha_1[\gamma_1])\alpha_1[u]-\alpha_1[\delta_1]\beta_1[u]=\int_0^1\alpha_1[k_1]g_1(s)f_1\left(u(s),v(s),\frac{|u'(s)|}{|r'(s)|},\frac{|v'(s)|}{|r'(s)|}\right)ds+\lambda \alpha_1[1],\\
&-\beta_1[\gamma_1]\alpha_1[u]+(1-\beta_1[\delta_1])\beta_1[u]=\int_0^1\beta_1[k_1]g_1(s)f_1\left(u(s),v(s),\frac{|u'(s)|}{|r'(s)|},\frac{|v'(s)|}{|r'(s)|}\right)ds+\lambda \beta_1[1].\\
\end{align*}
Therefore
\begin{equation*}
N_1 \left(
\begin{array}{ccccc}
\alpha_1[u] \\
 \beta_1[u]
\end{array}
\right)=\left(
\begin{array}{ccccc}
\displaystyle\int_0^1\alpha_1[k_1]g_1(s)f_1\left(u(s),v(s),\frac{|u'(s)|}{|r'(s)|},\frac{|v'(s)|}{|r'(s)|}\right)ds+\lambda \alpha_1[1]\\
\displaystyle\int_0^1\beta_1[k_1]g_1(s)f_1\left(u(s),v(s),\frac{|u'(s)|}{|r'(s)|},\frac{|v'(s)|}{|r'(s)|}\right)ds+\lambda \beta_1[1]
\end{array}
\right).
\end{equation*}
If we apply the matrix $(N_1)^{-1}$ to both sides of the last equality, then we obtain
\begin{align*}
\left(
\begin{array}{ccccc}
\alpha_1[u] \\
 \beta_1[u]
\end{array}
\right)&=(N_1)^{-1}\left(
\begin{array}{ccccc}
\displaystyle\int_0^1\alpha_1[k_1]g_1(s)f_1\left(u(s),v(s),\frac{|u'(s)|}{|r'(s)|},\frac{|v'(s)|}{|r'(s)|}\right)ds+\lambda \alpha_1[1]\\
\displaystyle\int_0^1\beta_1[k_1]g_1(s)f_1\left(u(s),v(s),\frac{|u'(s)|}{|r'(s)|},\frac{|v'(s)|}{|r'(s)|}\right)ds+\lambda \beta_1[1]
\end{array}
\right)\\
&\geq(N_1)^{-1}\left(
\begin{array}{ccccc}
\displaystyle\int_0^1\alpha_1[k_1]g_1(s)f_1\left(u(s),v(s),\frac{|u'(s)|}{|r'(s)|},\frac{|v'(s)|}{|r'(s)|}\right)ds\\
\displaystyle\int_0^1\beta_1[k_1]g_1(s)f_1\left(u(s),v(s),\frac{|u'(s)|}{|r'(s)|},\frac{|v'(s)|}{|r'(s)|}\right)ds
\end{array}
\right).
\end{align*}
Thus, as in previous step, we have
\begin{align*}
u(t)&\geq\frac{1}{D_1}\left[\Bigl[\gamma_1(t)(1-\beta_1[\delta_1])+\delta_1(t)\beta_1[\gamma_1]\Bigr]\int_0^1\alpha_1[k_1]g_1(s)f_1\left(u(s),v(s),\frac{|u'(s)|}{|r'(s)|},\frac{|v'(s)|}{|r'(s)|}\right)ds \right.\\
&+\left.\Bigl[\gamma_1(t)\alpha_1[\delta_1]+\delta_1(t) (1-\alpha_1[\gamma_1])\Bigr]\int_0^1\beta_1[k_1]g_1(s)f_1\left(u(s),v(s),\frac{|u'(s)|}{|r'(s)|},\frac{|v'(s)|}{|r'(s)|}\right)ds\right]\\
&+\int_0^1k_1(t,s)g_1(s)f_1\left(u(s),v(s),\frac{|u'(s)|}{|r'(s)|},\frac{|v'(s)|}{|r'(s)|}\right)ds+\lambda.\\
\end{align*}
Then, for $t \in [a_1,b_1]$, we obtain
\begin{align*}
u(t)&\geq\frac{1}{D_1}\left[\Bigl[\gamma_1(t)(1-\beta_1[\delta_1])+\delta_1(t)\beta_1[\gamma_1]\Bigr]\int_{a_1}^{b_1}\alpha_1[k_1]g_1(s)f_1\left(u(s),v(s),\frac{|u'(s)|}{|r'(s)|},\frac{|v'(s)|}{|r'(s)|}\right)ds \right.\\
&+\left.\Bigl[\gamma_1(t)\alpha_1[\delta_1]+\delta_1(t) (1-\alpha_1[\gamma_1])\Bigr]\int_{a_1}^{b_1}\beta_1[k_1]g_1(s)f_1\left(u(s),v(s),\frac{|u'(s)|}{|r'(s)|},\frac{|v'(s)|}{|r'(s)|}\right)ds\right]\\
&+\int_{a_1}^{b_1}k_1(t,s)g_1(s)f_1\left(u(s),v(s),\frac{|u'(s)|}{|r'(s)|},\frac{|v'(s)|}{|r'(s)|}\right)ds+\lambda\\
&\geq\inf_{A_1^{s_1,s_2}}f_1(w_1,w_2,z_1,z_2) \left[\frac{1}{D_1}\left[\Bigl[\gamma_1(t)(1-\beta_1[\delta_1])+\delta_1(t)\beta_1[\gamma_1]\Bigr]\int_{a_1}^{b_1}\alpha_1[k_1]g_1(s)ds \right.\right.\\
&+\left.\Bigl[\gamma_1(t)\alpha_1[\delta_1]+\delta_1(t) (1-\alpha_1[\gamma_1])\Bigr]\int_{a_1}^{b_1}\beta_1[k_1]g_1(s)ds\right]\\
&\left.+\int_{a_1}^{b_1}k_1(t,s)g_1(s)ds\right]+\lambda.
\end{align*}
Taking the minimum over $[a_1,b_1]$ gives
\begin{align*}
s_1&\geq\inf_{A_1^{s_1,s_2}}f_1(w_1,w_2,z_1,z_2) \left[\frac{1}{D_1}\left[\Bigl[a_1(1-\beta_1[\delta_1])+a_1\|\delta_1\|_{\infty}\beta_1[\gamma_1]\Bigr]\int_{a_1}^{b_1}\alpha_1[k_1]g_1(s)ds \right.\right.\\
&+\left.\Bigl[a_1\alpha_1[\delta_1]+a_1\|\delta_1\|_{\infty} (1-\alpha_1[\gamma_1])\Bigr]\int_{a_1}^{b_1}\beta_1[k_1]g_1(s)ds\right]\\
&\left.+\inf_{t \in [a_1,b_1]}\int_{a_1}^{b_1}k_1(t,s)g_1(s)ds\right]+\lambda >M_1s_1\frac{1}{M_1}+ \lambda,
\end{align*}
 a contradiction.
\end{proof}
By means of Theorem~\ref{ellyptic} and the fixed point index properties in Theorem~\ref{index}, we can state results on the existence of  \emph{multiple} positive solutions
for the system~\eqref{do2}. Here we enunciate a result on the existence of two positive solutions (see the papers~\cite{lan-lin-na,  lanwebb} for the conditions that assure three  o more positive results).
\begin{thm}\label{ellyptic2}
Suppose that  there exist $\rho _{i},s _{i},\theta_i\in (0,\infty )$ with $\rho _{i}/c_i<s_i <\theta_{i}$, $i=1,2$,  such that
\begin{align*}
&\inf_{A_i^{\rho_1,\rho_2}}
f_i(w_1,w_2,z_1,z_2)>M_i\rho_i,\\
&\sup_{\Omega^{s_1,s_2}}
f_i(w_1,w_2,z_1,z_2)<C_i,s_i\\
&\inf_{A_i^{\theta_1,\theta_2}}
f_i(w_1,w_2,z_1,z_2)>M_i\theta_i.
\end{align*}
Then the system~\eqref{do2} has at least two positive radial
solutions.
\end{thm}

\begin{ex} We note that Theorem ~\ref{ellyptic} and Theorem ~\ref{ellyptic2} can be apply when the nonlinearities $f_i$ are of the type
\begin{equation*}
f_i(u,v,|\nabla u|, |\nabla v|)=(\delta_i u^{\alpha_i}+\gamma_i v^{\beta_i})q_i(u,v,|\nabla u|, |\nabla v|)
\end{equation*}
with $q_i$ continuous functions bounded  by a strictly positive constant, $\alpha_i,\beta_i>1$ and suitable $\delta_i,\gamma_i \geq 0$.\\
For example, we can consider  in $\mathbb{R}^3$  the system of BVPs
\begin{gather}\label{ellbvpex}
\begin{cases}
&-\Delta u = \frac{1}{ |x|^4}\,(2-\sin(|\nabla u|^2+|\nabla v|^2)\,u^5 \text{ in } \Omega, \\
&-\Delta v= \frac{1}{ |x|^4}\frac{1}{\pi}\,\arctan\left(1+|\nabla u|^2+|\nabla v|^2\right)\,v^5\text{ in } \Omega,\\
&\displaystyle \lim_{|x|\to\infty}u(x)=u\left(\frac{1}{4}\right),\, \,\,\,\displaystyle 2 u-4\frac{\partial u}{\partial r}=u\left(\frac{1}{2}\right)\,\,\text{ for }|x|=1,\\
&\displaystyle\lim_{|x|\to\infty}v(x)=v\left(\frac{1}{4}\right),\,\,15v- \frac{\partial v}{\partial r}=5v\left(\frac{1}{2}\right)\,\,\text{ for }|x|=1.
\end{cases}
\end{gather}
Let $[a_1,b_1]=[a_2,b_2]=\left[\frac{1}{4},\frac{1}{2}\right]$;  by direct computation, we obtain $$D_i=\frac{c_i-1}{4(c_i+d_i)};$$
\begin{align*}
&\sup_{t\in[0,1]}\int_{0}^{1}k_i(t,s)g_i(s)ds=\frac{(c_i+2d_i)^2}{8(c_i+d_i)^2};\,\,\inf_{t\in\left[\frac{1}{4},\frac{1}{2}\right]}\int_{\frac{1}{4}}^{\frac{1}{2}}k_i(t,s)g_i(s)ds=\frac{5c_i+8d_i}{128 (c_i + d_i)};\\
&\int_0^1\alpha_i[k_i]g_i(s)ds=\frac{(3c_i+7d_i)}{32(c_i+d_i)}; \,\,\, \int_0^1\beta_i[k_i]g_i(s)ds=\frac{(c_i+3d_i)}{8(c_i+d_i)};\\
&\int_{a_i}^{b_i}\alpha_i[k_i]g_i(s)ds=\frac{5c_i+8d_i}{128 (c_i + d_i)};\,\,\int_{a_i}^{b_i}\beta_i[k_i]g_i(s)ds=\frac{3}{32 }\left(1 - \frac{c_i}{2 (c_i + d_i)}\right).
\end{align*}

Since in our example the mixed perturbed conditions stated that $c_1=2, d_1=4, c_2=3, d_2=\frac{1}{5}$ we easily compute $C_i$ and $M_i$ that becomes
\begin{align*}
C_1=\frac{9}{47},\,\, C_2=\frac{10240}{9561}, \,\, M_1=\frac{96}{41},\,\, M_2=\frac{10240}{1209}.
%&\displaystyle C_{i}:=\left[\frac{4(c_i+d_i)}{c_i-1}\Bigg\{\left[\left(1-\frac{1}{2(c_i+d_i)}\right)+\frac{1}{c_i+d_i}\left(1-\frac{c_i}{2(c_i+d_i)}\right)\right]\frac{(3c_i+7d_i)}{32(c_i+d_i)}\right. \\
%&\left. \nonumber\displaystyle+\left[\frac{1}{4(c_i+d_i)}+ \frac{c_i}{4(c_i+d_i)^2}\right]\frac{(c_i+3d_i)}{8(c_i+d_i)}\Bigg\}
%+\frac{(c_i+2d_i)^2}{8(c_i+d_i)^2}\right]^{-1}\\
%%&=\left[\frac{1}{1-c_i}\Bigg\{\frac{(2c_i^2+2d_i^2+4c_id_i+d_i)(3c_i+7d_i)}{16(c_i+d_i)^2}+\frac{(2c_i+d_i)(c_i+3d_i)}{8(c_i+d_i)^2}\Bigg\}
%%+\frac{(c_i+2d_i)^2}{8(c_i+d_i)^2}\right]^{-1}\\
%%&=\left[\frac{6 c_i^3 + 17 c_i d_i (1 + 2 d_i) + d_i^2 (13 + 14 d_i) + c_i^2 (4 + 26 d_i))}{16(1-c_i)(c_i+d_i)^2}+\frac{(c_i+2d_i)^2}{8(c_i+d_i)^2}\right]^{-1}\\
%&=\frac{16(c_i-1)(c_i+d_i)^2}{8 c_i^3 + 3 c_i d_i (3 + 14 d_i) + d_i^2 (5 + 14 d_i) + c_i^2 (2 + 34 d_i)}.
\end{align*}
Choosing $\rho_1=\frac{1}{3},\,\,\rho_2=\frac{1}{2},\,\, s_1=2,\,\,s_2=3$, one can note that, immediately, $\rho_i<\frac{1}{4}s_i$, for all $i=1,2$ and, moreover
\begin{align*}
&\sup_{\Omega^{\rho_1,\rho_2}}\, f_1\,\leq 3\rho_1^5=  \frac{1}{243}<\frac{3}{47}=C_1\rho_1, \,\,\,\,\, \inf_{A_1^{s_1,s_2}}  f_1\,\geq  s_1^5=32>\frac{192}{41}=M_1s_1,\\
&\sup _{\Omega^{\rho_1,\rho_2}} f_2\; \leq \frac{\rho_2^5}{2}=\frac{1}{64}<\frac{5120}{9561}=C_2\rho_2, \,\,\,\,\, \inf_{A_2^{s_1,s_2}}  f_2\,\geq  \frac{s_2^5}{4}=\frac{243}{4}>\frac{10240}{403}=M_2s_2,
\end{align*}
where supremum and infimum are computed on
\begin{align*}
&\Omega^{\frac{1}{3},\frac{1}{2}}= \left [0,\frac{1}{3}\right]\times\left[0,\frac{1}{2}\right]\times [0,+\infty)\times[0,+\infty);\\
&A_1^{2,3}= \left[2,8\right]\times[0,12]\times[0,+\infty)\times[0,+\infty);\\
&A_2^{2,3}= \left[0,8\right]\times [3,12]\times[0,+\infty)\times[0,+\infty).
\end{align*}
Then the hypotheses of Theorem~\ref{ellyptic} are
satisfied and hence the system~\eqref{ellbvpex} has at least one
positive solution.
\end{ex}

\section{Nonexistence results}
We now show a nonexistence result for the system of elliptic
equations \eqref{do2} when the the functions $f_i$ have an enough "small" or "large" growth.

\begin{thm}
Assume that one of following conditions holds:
\begin{equation}\label{cond1}
f_i(w_1,w_2,z_1,z_2)<C_i w_i\,,\,
w_i>0, \text{ for }i=1,2,
\end{equation}
\begin{equation}\label{cond2}
f_i(w_1,w_2, z_1,z_2)>M_i w_i\,\,,\
w_i>0, \text{ for }i=1,2.
\end{equation}
Then the only possible positive solution of the system \eqref{do2} is the zero one.
\end{thm}

\begin{proof}
Suppose that (\ref{cond1}) holds and assume that there exists a
solution $(\bar{u},\bar{v})$ of \eqref{do2}, $(\bar{u},\bar{v})\neq
(0,0)$; then $(u,v):=(\bar{u}\circ r,\bar{v}\circ r)$ is a fixed
point of $T$. Let, for example,  $\|(u,v)\|=\|u\| \leq\ 4\|u\|_{\infty}\neq 0$.\\
 Then, for $t\in [0,1]$, proceeding as in the proof of Theorem~\ref{ellyptic}, we have
 \begin{align*}
u(t)
&=\frac{1}{D_1}\left[\Bigl[\gamma_1(t)(1-\beta_1[\delta_1])+\delta_1(t)\beta_1[\gamma_1]\Bigr]\int_0^1\alpha_1[k_1]g_1(s)f_1\left(u(s),v(s),\frac{|u'(s)|}{|r'(s)|},\frac{|v'(s)|}{|r'(s)|}\right)ds \right.\\
&+\left.\Bigl[\gamma_1(t)\alpha_1[\delta_1]+\delta_1(t) (1-\alpha_1[\gamma_1])\Bigr]\int_0^1\beta_1[k_1]g_1(s)f_1\left(u(s),v(s),\frac{|u'(s)|}{|r'(s)|},\frac{|v'(s)|}{|r'(s)|}\right)ds\right]\\
&+\int_0^1k_1(t,s)g_1(s)f_1\left(u(s),v(s),\frac{|u'(s)|}{|r'(s)|},\frac{|v'(s)|}{|r'(s)|}\right)ds\\
&<C_1\left(\frac{1}{D_1}\left[\Bigl[(1-\beta_1[\delta_1])+\|\delta_1\|_{\infty}\beta_1[\gamma_1]\Bigr]\int_0^1\alpha_1[k_1]g_1(s)u(s)ds \right.\right.\\
&+\left.\Bigl[\alpha_1[\delta_1]+\|\delta_1\|_{\infty} (1-\alpha_1[\gamma_1])\Bigr]\int_0^1\beta_1[k_1]g_1(s)u(s)ds\right]
\left.+\int_0^1k_1(t,s)g_1(s)u(s)ds\right)\\
&\leq C_1\, \|u\|_{\infty}\left(\frac{1}{D_1}\left[\Bigl[(1-\beta_1[\delta_1])+\|\delta_1\|_{\infty}\beta_1[\gamma_1]\Bigr]\int_0^1\alpha_1[k_1]g_1(s)ds \right.\right.\\
&+\left.\Bigl[\alpha_1[\delta_1]+\|\delta_1\|_{\infty} (1-\alpha_1[\gamma_1])\Bigr]\int_0^1\beta_1[k_1]g_1(s)ds\right]
\left.+\int_0^1k_1(t,s)g_1(s)ds\right)
\end{align*}
 For $u>0$, taking the supremum for $t\in [0,1]$, we have $\|u\|_\infty<\|u\|_\infty$, a contradiction.\\
Suppose that (\ref{cond2}) holds and assume that there exists
$(u,v)\in K$ such that $(u,v)=T(u,v)$ and $(u,v)\neq (0,0)$. Let,
for example, $\|u\|_{\infty}\neq 0$; then
$\sigma:=\displaystyle\min_{t\in[a_1,b_1]}u(t)>0$ since $u \in K_1$. 
Thus, as in the proof of Theorem~\ref{ellyptic}, we have, for $t \in [a_1,b_1]$, 
\begin{align*}
u(t)&\geq\frac{1}{D_1}\left[\Bigl[a_1(1-\beta_1[\delta_1])+a_1\|\delta_1\|_{\infty}\beta_1[\gamma_1]\Bigr]\int_{a_1}^{b_1}\alpha_1[k_1]g_1(s)f_1\left(u(s),v(s),\frac{|u'(s)|}{|r'(s)|},\frac{|v'(s)|}{|r'(s)|}\right)ds \right.\\
&+\left.\Bigl[a_1\alpha_1[\delta_1]+a_1\|\delta_1\|_{\infty} (1-\alpha_1[\gamma_1])\Bigr]\int_{a_1}^{b_1}\beta_1[k_1]g_1(s)f_1\left(u(s),v(s),\frac{|u'(s)|}{|r'(s)|},\frac{|v'(s)|}{|r'(s)|}\right)ds\right]\\
&+\int_{a_1}^{b_1}k_1(t,s)g_1(s)f_1\left(u(s),v(s),\frac{|u'(s)|}{|r'(s)|},\frac{|v'(s)|}{|r'(s)|}\right)ds\\
&>M_1\left(\frac{1}{D_1}\left[\Bigl[a_1(1-\beta_1[\delta_1])+a_1\|\delta_1\|_{\infty}\beta_1[\gamma_1]\Bigr]\int_{a_1}^{b_1}\alpha_1[k_1]g_1(s)u(s)ds \right.\right.\\
&+\left.\Bigl[a_1\alpha_1[\delta_1]+a_1\|\delta_1\|_{\infty} (1-\alpha_1[\gamma_1])\Bigr]\int_{a_1}^{b_1}\beta_1[k_1]g_1(s)u(s)ds\right]
\left.+\int_{a_1}^{b_1}k_1(t,s)g_1(s)u(s)ds\right)\\
&\geq M_1\,\sigma \left(\frac{1}{D_1}\left[\Bigl[a_1(1-\beta_1[\delta_1])+a_1\|\delta_1\|_{\infty}\beta_1[\gamma_1]\Bigr]\int_{a_1}^{b_1}\alpha_1[k_1]g_1(s)ds \right.\right.\\
&+\left.\Bigl[a_1\alpha_1[\delta_1]+a_1\|\delta_1\|_{\infty} (1-\alpha_1[\gamma_1])\Bigr]\int_{a_1}^{b_1}\beta_1[k_1]g_1(s)ds\right]
\left.+\int_{a_1}^{b_1}k_1(t,s)g_1(s)ds\right).
\end{align*}

 For $u>0$, taking the infimum   for $t\in [a_1,b_1]$, we obtain $\sigma>\sigma$, a contradiction.
\end{proof}

\end{document}